\documentclass[a4paper,10pt]{amsart}

\usepackage{latexsym}
\usepackage{amscd}
\usepackage{graphics}
\usepackage{amsmath}
\usepackage{amssymb}
\usepackage{mathrsfs}

\newcommand{\R}{\mathbb R}

\newcommand{\N}{\mathbb N}
\newcommand{\sphere}[1]{\mathbb{S}^{#1}}
\newcommand{\e}{\mathrm e}
\newcommand{\abs}[1]{\left\vert #1 \right\vert}
\newcommand{\nor}[1]{\left\Vert #1 \right\Vert}

\newcommand{\D}{\mathcal D}
\newcommand{\U}{\mathcal U}
\newcommand{\eps}{\varepsilon}

\newtheorem{thm}{\bf Theorem}[section]      % numbered within each section
\newtheorem{cor}[thm]{\bf Corollary}        % numbered along with Theorem
\newtheorem{lem}[thm]{\bf Lemma}            % numbered along with Theorem
\newtheorem{prop}[thm]{\bf  Proposition}     % numbered along with Theorem
\newtheorem{defn}[thm]{\bf Definition}      % numbered along with Theorem
\newtheorem{assumptions}[thm]{\bf Assumptions}      % numbered along with Theorem
\newtheorem{rem}[thm]{\bf Remark}       % numbered along with Theorem

\begin{document}

\title[Harmonic functions in union of chambers]
{Harmonic functions in union of chambers.}
\author{Laura Abatangelo, Susanna Terracini}

\thanks{
\indent Partially supported by the PRIN2009 grant ``Critical Point Theory and
Perturbative Methods for Nonlinear
\indent Differential Equations''.}

\address{Laura Abatangelo: Dipartimento di Matematica e Applicazioni,
Universit\`a di Milano Bicocca, Piazza Ateneo Nuovo, 1, 20126 Milano (Italy)}
\email{laura.abatangelo@unimib.it}
\address{Susanna Terracini: Dipartimento di Matematica ``Giuseppe Peano'', Via Carlo Alberto 10, 10123 Torino (Italy)}
\email{susanna.terracini@unito.it}

\date{\today}
\maketitle

\begin{abstract}
We characterize the set of   harmonic functions with Dirichlet boundary conditions in unbounded
domains which are union of several different chambers. We analyze the asymptotic behavior of the solutions in connection with
the changes in the domain's geometry. Finally we classify all (possibly sign-changing) infinite energy solutions having given asymptotic frequency  at
the infinite ends of the domain.
\end{abstract}

\section{Introduction}

In this paper we are concerned with solutions to the
following problem
\begin{equation}\label{problema_omega}
\left\{ \begin{array}{ll} \Delta u=0 &\text{in $\Omega$} \\
u=0 &\text{on $\partial \Omega$,}
\end{array}\right.
\end{equation}
where $\Omega$ is a particular unbounded domain defined as the union
of two or more infinite cylinders. In this context the term \emph{chamber} stands
exactly for cylinder. We became interested in these issues in connection with the
problem of the interplay of the geometry of the domain with the transmission of frequencies of solutions,
as it will appear in the sequel. As a matter of
facts, problems of type \eqref{problema_omega} may arise, for example,
from a blow-up analysis for eigenvalues equations in bounded
domains with varying geometries.
This type of equations may describe the possible transmission of frequency from
a chamber to another one, when passing through a certain number of other
chambers, connected by thin tubes (whose section is negligible with respect to its own
length), (see e.g. \cite{FT12,AFT12}).

As this is the simplest case where the domain
presents a sensitive change of geometry, one may
expect  the domain's geometry and solutions' shape to be strictly
related to each other. We mean that such geometric
changes in the domain affects the solutions' shape as well as,
from the opposite point of view, that solutions may carry some
information about the domain's geometry.

We are now going to specify the context and the notation that we will use throughout the paper.
Let $U^R$ and $U^L$ two open regular connected domains in $\R^{N-1}$
for $N\geq2$, possibly unbounded and let
\begin{eqnarray*}
C^R &:=& \{(x,y)\in \R\times\R^{N-1} \ \text{s.t.}\ x>0\ \text{and}\ y\in U^R\};\\
C^L &:=& \{(x,y)\in \R\times\R^{N-1} \ \text{s.t.}\ x<0\ \text{and}\ y\in U^L\};\\
\Omega &:=& C^R \cup C^L \cup \Gamma \qquad \text{being}\ \Gamma:= \partial C^R\cap\partial C^L.
\end{eqnarray*}
% Further, being $v$ any $\mathcal C^1$ solution of \eqref{problema_omega} with finite energy on the left end, that is $C^L$, 
% we define its Almgren frequency function:
% \begin{equation}\label{enne_intro}
%  N(v)(x):= \dfrac{\int_{\Omega_x}\abs{\nabla v}^2}{\int_{\Gamma_x}v^2},
% \end{equation}
% where $\Omega_x:=\{(\xi,\eta)\in \Omega :\xi<x \}$ and $\Gamma_x:=\{(x,y):\ y\in U^R\}$.
% Analogously, for any $\mathcal C^1$ solution of \eqref{problema_omega} with finite energy on the right end, that is $C^R$,
% we define its Almgren frequency function as in \eqref{enne_intro}
% where $\Omega_x:=\{(\xi,\eta)\in \Omega :\xi>x \}$ and $\Gamma_x:=\{(x,y):\ y\in U^R\}$.
% We set $N(0)(x)\equiv 0$ for the trivial solution to \eqref{problema_omega}.
% %
% The ratio in \eqref{enne_intro} is acquired from the well-known Almgren frequencies, 
% which were introducted by Almgren in the '70s to study certain properties of harmonic functions
% and from then on they were employed in many other branches of the analysis of pdes.
We stress that in our setting positive solutions can not have finite energy at both ends of the domain. 
In the same way, uniqueness of solutions of inhomogoneous Laplace equations does not hold, unless the energy is supposed to be finite.
Therefore, in order to classify solutions of  \eqref{problema_omega},
we need to waive the energy boundedness and to allow infinite energy solutions.
As it will appear in the proofs, we can handle infinite energy solutions by imposing suitable thresholds to the so-called 
Almgren quotient.
%defined in \eqref{enne_intro}. 
Being $v$ any solution of \eqref{problema_omega} 
we define its Almgren frequency function:
\begin{equation}\label{enne_intro}
 N(v)(x):= \dfrac{\int_{\Omega_x}\abs{\nabla v}^2}{\int_{\Gamma_x}v^2},
\end{equation}
where $\Gamma_x:=\{(x,y):\ y\in U^R\}$,  
$\Omega_x:=\{(\xi,\eta)\in \Omega :\xi\in (0,x) \}$ if $x>0$ whereas 
$\Omega_x:=\{(\xi,\eta)\in \Omega :\xi\in(x,0) \}$ if $x<0$.
We set $N(0)(x)\equiv 0$ for the trivial solution to \eqref{problema_omega}.
The ratio in \eqref{enne_intro} is acquired from the well-known Almgren frequencies, 
which were introducted by Almgren in the '70s to study certain properties of harmonic functions
and from then on they were employed in many other branches of the analysis of pdes.

On the connecting hyperplane between the two chambers, we are considering the following eigenvalue problems:
\begin{equation}\label{eigenvalue_eq_connecting_section}
(i) \begin{cases}
   -\Delta \psi_k^L = \lambda_k^L \psi_k^L &\text{on }U^L\\
 \psi_k^L = 0 &\text{on }\partial U^L,
 \end{cases}
 \qquad
(ii) \begin{cases}
   -\Delta \psi_k^R = \lambda_k^R \psi_k^R &\text{on }U^R\\
 \psi_k^R = 0 &\text{on }\partial U^R; 
 \end{cases}
\end{equation}
here $\Delta$ denotes the $(N-1)$-Laplacian over $U^L$ and $U^R$ respectively.

For what concerns our first aim about possible characterization of solutions, our main result 
relies on the following lemma
\begin{lem}
 Let $u$ be any nontrivial solution to the problem \eqref{problema_omega}. Then there exist
\begin{align*}
 \lim_{x\to +\infty} N(x) = l^R \in  \left\{ \sqrt{\lambda_j^R} \right\}_{j=1}^{+\infty} \cup \{+\infty\} \\
 \lim_{x\to -\infty} N(x) = l^L \in  \left\{ \sqrt{\lambda_j^L} \right\}_{j=1}^{+\infty} \cup \{+\infty\}.
\end{align*}
\end{lem}

We fix two numbers $d^{R} \in \left\{ \lambda_j^R \right\}_{j=1}^{+\infty}$ 
and $d^{L} \in \left\{ \lambda_j^L \right\}_{j=1}^{+\infty}$ 
and define the following set 
\begin{align}\label{S}
 \mathcal S = \left\{ u \ \mathcal C^1\text{ solution to \eqref{problema_omega} such that } 
\text{if } \int_{C^{R,L}} \abs{\nabla u}^2 =+\infty \text{ then } l^{R,L}\leq \sqrt{d^{R,L}}  
%\text{and } \text{if } \int_{C^L} \abs{\nabla u}^2 =+\infty \text{ then } l^L\leq d^L 
\right\}.
\end{align}
Now we can state our main result as follows
\begin{thm}\label{teo_sol_freq_lim_sign_changing}
% For any $\alpha$ and $\beta$ in $\N$, let $\lambda_\alpha^R$ and $\lambda_\beta^L$ be the $\alpha$-th and 
% $\beta$-th eigenvalue of the Dirichlet $(N-1)$-Laplacian on $U^R$ and $U^L$ respectively, and let us
% denote $m_\alpha^R$ and $m_\beta^L$ their respective multiplicity (see problems in Equation
% \eqref{eigenvalue_eq_connecting_section}). 
% Then, for every $j,k \in \N$, the functions' set
% \[ \mathcal S_{j,k} = \left\{ u \ \mathcal C^1\text{ solution to \eqref{problema_omega} and }
% \lim_{x\to+\infty}N(x)\leq \sqrt{\lambda_k^R} \text{ or }
% \lim_{x\to-\infty}N(x)\leq \sqrt{\lambda_j^L} \right\} \] is a linear space
% of dimension $\sum_{\alpha=1}^{k}m_\alpha^R + \sum_{\beta=1}^{j} m_\beta^L$.
The set $\mathcal S$ defined in \eqref{S} is a linear space of dimension 
\[ \text{dim}\mathcal S = m(d^R) + m(d^L) \]
where $m$ denotes the Morse index.

In particular, if we restrict to those solutions with finite energy on one hand of the domain, namely $C^L$,  
the set 
\begin{equation}\label{SL}
\mathcal S_L = \left\{ u \ \mathcal C^1 \text{ solution to \eqref{problema_omega} with } 
\int_{C^L} \abs{\nabla u}^2 <+\infty
\text{ and } l^R \leq \sqrt{d^R} \right\} 
\end{equation}
is a linear space of dimension $m(d)$, where $m$ denotes again the Morse index. 
\end{thm}

For the reader's convenience we recall that in this context the Morse index of the eigenvalue $\lambda_k^{L,R}$ 
is the sum of the multiplicity of the eigenvalues $\lambda_j^{L,R}$ with $j\leq k$.

If we focus our attention just on positive solutions, 
as a byproduct of the previous result we obtain the following theorem, which corresponds to
the particular case $d^R=\lambda_1^R$ in Equation \eqref{SL}:
\begin{thm}\label{teorema_omega}
There exists a unique  (up to a multiplicative constant)  positive $\mathcal C^1$ solution $v^L$ to the
problem \eqref{problema_omega}, provided
\begin{equation}\label{finite_energy}
\int_{C_L}\abs{\nabla v^L}^2 < \infty.
\end{equation}
Moreover,
\begin{itemize}
\item if $U^R$ is bounded, then $v^L$ is asymptotic to
$\e^{\sqrt{\lambda_1}x}\psi_1(y)$ as $x\rightarrow\infty$ uniformly
with respect to $y\in U^R$, being $\psi_1$ and $\lambda_1$ the first eigenfunction and eigenvalue respectively to the problem
(ii) in Equation \eqref{eigenvalue_eq_connecting_section} with a little abuse of notation;
\item if $C^R$ is the whole right halfspace of $\R^N$, then $v^L$ is asymptotic to
$x$ as $x\rightarrow\infty$ uniformly with respect to $y\in
\R^{N-1}$.
\end{itemize}
An analogous statement defines $v^R$. Finally, all positive solutions to problem \eqref{problema_omega}
are positive convex combinations of $v^L$ and $v^R$.
\end{thm}
\begin{rem}
 As appearing in Theorem \ref{teo_sol_freq_lim_sign_changing}, the dimension of the space $\mathcal S_L$ 
even in Theorem \ref{teorema_omega} is related to 
the multiplicity of the first eigenvalue $\lambda_1^R$. Therefore, we stress the uniqueness stated in Theorem 
\ref{teorema_omega} relies essentially on the assumption $U^R$ is a \emph{connected} domain in $\R^{N-1}$.  
\end{rem}

In the framework of positive solutions, Theorem \ref{teorema_omega} is not a brand new result, 
as we can find it within the so-called \emph{General Martin Theory}.
This is a quite general theory which provides a one-to-one correspondence 
between regular \emph{positive} solutions and the points of the so-called \emph{minimal Martin boundary}
by means of finite harmonic measures supported on the minimal Martin boundary. 
We then foresee that
the Martin boundary is useful to gain some information about the number of the
(linearly independent) positive solutions
to a differential equation, whenever the differential operator satisfies several minimal assumptions. 
We then find worthwhile recalling some general concepts of the 
\emph{General Martin Theory}. To this aim we refer to the book by Ross Pinsky \cite{Pinsky},
chapter 7.

In order to state the known results, let us consider a quite general differential operator $L$ on a domain $D \subseteq\R^N $
satisfying the following
\begin{assumptions}\label{assumptions_martin}
For any $D'\subset\subset D$ the operator $L$ is of the form
$ L= \frac12 \nabla \cdot a\nabla + b\cdot \nabla +V$, with
$a_{i,j},\ b_i \in \mathcal C^{1,\alpha}(\overline{D'})$,
$V\in \mathcal C^\alpha(D')$ and
$\sum_{i,j} (x)v_i v_j >0$ for all $v\in\R^N\setminus\{0\}$ and for all $ x\in D'$.
\end{assumptions}
It is defined
\[\mathcal P_L(D)=\{ u\in \mathcal C^{2,\alpha}(D):\ Lu=0 \text{ and }u>0 \text{ in }D \}\]
the set of all regular positive solutions; and, fixed a point $x_0 \in D$ and, denoting $G$  the Green's function, define the \emph{Martin kernel} as
\[ k(x,y) =
\begin{cases}
\dfrac{G(x,y)}{G(x_0,y)} & y\neq x,\ y\neq x_0\\
0 & y=x_0,\ x\neq x_0\\
1 & y=x=x_0.
\end{cases}
\]

\begin{defn}
A sequence $\{ y_n \}_n \subset D$ for which the limit $\lim_{n\to\infty}k(x,y_n) \in \mathcal P_L(D)$
is called a \emph{Martin sequence}. Two Martin sequences which have the same limit are called \emph{equivalent}.
The collection of such equivalence classes is called the \emph{Martin boundary for $L$ on $D$}.
\end{defn}
We briefly mention that the Martin boundary does not depend on the choice of the fixed point
$x_0$ in the Martin kernel and it can be endowed with a suitable topology; we do not enter into the details,
since they go beyong our specific aim. More related to our work, we find the following

\begin{defn}\label{def_minimality_martin}
A function $u\in \mathcal P_L(D)$ is called \emph{minimal} if whenever $v\in \mathcal P_L(D)$ and $v\leq u$ then in fact
$v=cu$ for some constant $c\in(0,1]$.

Given a point $\xi$ of the Martin boundary, the notation $k(x;\xi)$ means that, up to positive multiples,
\[ k(x;\xi)= \lim_{n\to +\infty} k(x;y_n) \]
where $y_n$ is any representative of the equivalence class $\xi$.

A point $\xi$ on the Martin boundary is called a \emph{minimal Martin boundary point} if $k(x,\xi)$ is minimal.
\end{defn}

\begin{thm}[Martin Representation Theorem]
 Let $L$ satisfy the Assumptions \ref{assumptions_martin} on a domain $D\subseteq \R^N$ and assume that $L$
is subcritical. Then for each $u\in \mathcal P_L(D)$ there exists a unique finite measure $\mu_u$ supported on the minimal
Martin boundary $\Lambda_0$ such that
\[ u(x)= \int_{\Lambda_0} k(x,\xi)\mu_u(d\xi). \]
Conversely, for each finite measure $\mu$ supported on the minimal Martin boundary $\Lambda_0$,
\[ u(x):= \int_{\Lambda_0} k(x,\xi)\mu_u(d\xi) \ \in \mathcal P_L(D). \]
\end{thm}
As already mentioned this is the basic theorem in order to state a one-to-one correspondence
between the elements of the set $\mathcal P_L(D)$ and the points of the \emph{minimal Martin boundary}
by means of finite harmonic measures supported on the minimal Martin boundary. 
% It is worthwhile stating the following proposition, see \cite{Pinsky}:
%
% \begin{prop}\label{prop_martin_minimality}
% $u \in \mathcal P_L(D)$ is minimal if and only if $u(x)=k(x,\xi)$ for some $\xi$ in the minimal Martin boundary.
% \end{prop}

The General Martin Theory covers even our case of domains formed by different chambers by means of the following theorems
\begin{thm}[Theorem 6.6 in \cite{Pinsky}]
Let $D$ be a non-compact open $N$-dimensional $\mathcal C^{2,\alpha}$-Riemannian manifold with $m$ ends,
i.e. it can be represented in the form $D=F \cup E_1 \cup \ldots \cup E_m$ with $F$ is bounded and closed,
$E_i$ is open, $\overline E_i \cap \overline E_j = \emptyset$ for $i\neq j$ and $F \cap \overline E_i = \emptyset$.
Let $L$ on $D$ satisfy the assumptions \ref{assumptions_martin} and be subcritical.
Then the Martin boundary for $L$ decomposes into $m$ components in the following sense:
if $\{x_n\}_n \subset D$ is a Martin sequence, then all but a finite number of its terms lie in $E_i$
for some $i=1,\ldots,m$.
\end{thm}

\begin{cor}[Corollary 6.7 in \cite{Pinsky}]\label{cor6.7}
Let $L$ satisfy assumptions \ref{assumptions_martin} and be subcritical on the domain $D=(\alpha,\beta)$ where
$-\infty\leq \alpha < \beta\leq +\infty$. Then the Martin boundary of $L$ on $D$ consists of two points.
More specifically, a sequence $\{x_n\}_n$ with no accumulation points in $D$ is a Martin sequence if and only if
$\lim_{n\to\infty} x_n = \alpha$ or $\lim_{n\to\infty} x_n = \beta$.
\end{cor}
\noindent Moreover (see Proposition 5.1.3 in \cite{Pinsky}), in this case $\mathcal P_L(D)$ is 2-dimensional, which means that, via the
Martin Representation Theorem, the minimal Martin boundary consists
exactly of two points.

In particular, a suitable $N$-dimensional generalization of the
previous Corollary covers the case of Theorem \ref{teorema_omega} in the present paper. More
precisely, the Martin boundary of our domain $\Omega$ consists in its
topological boundary together with the union of two points, which
can be identified, roughly speaking, with the two ends of the
domain. Taking into account Dirichlet boundary conditions in our problem \eqref{problema_omega}, 
this means that problem \eqref{problema_omega} has got exactly two linearly independent positive solutions,
as we are able to show, too. 
Thus, our Theorem \ref{teorema_omega} does not provide any
additional information to the known results provided by the General
Martin Theory, except maybe by gaining greater understanding of the
positive solutions' space $\mathcal P_L(D)$: under Dirichlet
boundary conditions, a basis of $\mathcal P_L(D)$ is formed by two
positive regular functions which have finite energy on one end of
the domain, whereas on the other end they tend to infinity. Further,
we are able to describe exactly the divergent behavior, as well as
to prove every element in $\mathcal P_L(D)$ to be minimal according to Definition
\ref{def_minimality_martin}. Then, our original contribution does no
longer refer strictly to the result, but rather to the method: an
upper bound for the Almgren quotient of possible solutions is the
key ingredient for existence of solutions. In the case of positive
solutions the specific threshold for the Almgren
frequency is set by the positivity assumption of the solutions,
but the method can be extended even to sign-changing solutions, which are not included 
in the General Martin Theory, providing a stronger 
result that is Theorem \ref{teo_sol_freq_lim_sign_changing}.

As already mentioned, we are enforced to consider infinite energy solutions. 
As a second point of our work, we follow the idea that normalizing
their necessary divergent asymptotic behavior, we force the asymptotic (vanishing) behavior of the solution
even at the other hand of the domain. For a pair of cylinders, the rate of growth at $+\infty$ can be related
with the rate of vanishing at $-\infty$ by means of the evaluation of a transfer operator (see Section \S \ref{sec:transfer}).
As remarked in \S \ref{subsec:many_chambers}, composition of such transfer operators can be useful to handle a concatenation of many cylinders.
Finally, we shall classify all positive solutions and all infinite energy solutions having the smallest possible  growth at infinity.

\bigskip

The paper is organized as follows: in Section 2 we examine the
existence of positive solutions to problem \eqref{problema_omega}
when $\Omega$ is the cylinder $C^R$ either when its section is
bounded or when it is a whole hyperplane, and we investigate their
possible behavior at infinity; in Section 3 we collect the previous
results in order to prove Theorem \ref{teorema_omega} (see also Theorem \ref{teorema_omega_energia_finita_da_una_parte})
and Theorem \ref{teo_sol_freq_lim_sign_changing}. 
In Section 4 we study the relation between the asymptotic behavior of positive solutions at $+\infty$
and $-\infty$, generalizing our results to domains which are union
of more than two chambers in the very last subsection.

\section{Existence and uniqueness of a positive harmonic function on $C^R$.}

We claim the following
\begin{thm}\label{teorema_semicilindro}
There exists a unique  (up to a multiplicative constant)  positive
solution $u$ to the problem
\begin{equation}\label{problema_semicilindro}
\left\{ \begin{array}{ll} \Delta v=0 &\text{in $C^R$} \\
v=0 &\text{on $\partial C^R$}
\end{array}\right.
\end{equation}
if $U_R$ is bounded or it is a whole hyperplane.
In the first case it will be
\begin{equation}\label{v}
v(x,y)=\left(\e^{\sqrt{\lambda_1}x} - \e^{-\sqrt{\lambda_1}x}\right)\psi_1(y),
\end{equation}
being $\lambda_1$ and $\psi_1$ the first eigenvalue and the first eigenfunction respectively of the problem \eqref{eigenvalue_eq_connecting_section}
item (ii);
whereas in the second case it will be
\begin{equation}\label{x}
v(x,y)= x
\end{equation}
denoting $x$ the first variable in $\R^N$.
\end{thm}

\begin{rem}
We stress the aforementioned solutions have an infinite energy.
\end{rem}

In order to prove this theorem, we will study the two cases separately.

\subsection{The case $U^R$ bounded.}

It is quite simple to prove that the function $v$ defined in \eqref{v} is a solution to the problem \eqref{problema_semicilindro}.
Moreover, we stress it is asymptotic to $\e^{\sqrt{\lambda_1}x}\psi_1(y)$ as $x\rightarrow\infty$.
We aim to prove it is in fact the unique solution.

\begin{prop}\label{step1}
The function $v$ defined in \eqref{v} is the unique solution up to
multiplications by constants.
\end{prop}
The proof relies essentially on three different tools: the so-called
``Phragm\`{e}n-Lindel\"{o}f Principle'', which may be read as a
comparison principle on unbounded domains, a boundary version of the
Harnack inequality, and an Almgren--type argument. For similar arguments, see \cite{Pi88,Pi94}.

Let us recall the well-known Phragm\'{e}n--Lindel\"{o}f Principle
stated for the Laplace operator:
\begin{thm}[Phragm\'{e}n--Lindel\"{o}f Principle, \cite{PW}]\label{PL-principle}
 Let $D$ be a domain, bounded or unbounded, and let $u$ satisfy
\begin{eqnarray*}
 -\Delta u&\leq&0 \ \textrm{in $D$,}\\
 u&\leq&0\ \textrm{on $\Gamma$,}
\end{eqnarray*}
where $\Gamma$ is a subset of $\partial D$. Suppose that there is an
increasing sequence of bounded domains $D_1 \subset D_2 \subset
\cdots \subset D_k \subset\cdots$ with properties
\begin{enumerate}
 \item each $D_k$ is contained in $D$; for each point $x\in D$ there is an integer $N$
      such that $x\in D_N$;
 \item the boundary of each $D_k$ consists in two parts $\Gamma_k$ and ${\Gamma_k}'$
      where $\Gamma_k$ is a subset of $\Gamma$ and ${\Gamma_k}'$ is a subset of $D$.
\end{enumerate}
Further, suppose there exists a sequence $\{w_k\}$ which satisfies
\begin{eqnarray*}
 w_k(x)&>&0 \ \textrm{on $D_k\cup \partial D_k$,} \\
 -\Delta w_k &\geq0& \ \textrm{in $D_k$.}
\end{eqnarray*}
Assume there is a function $w(x)$ with the property that at each point $x\in D$ the inequality
$$ w_k(x)<w(x) $$
holds for all $k$ above a certain integer $N_x$. If $u$ satisfies the growth condition
$$\liminf_{k\rightarrow\infty} \left\{\sup_{{\Gamma_k}'}\dfrac{u(x)}{w_k(x)}\right\}\leq 0$$
then
$$u\leq0 \ \textrm{in $D$.}$$
\end{thm}

%\textit{Proof of Proposition \eqref{step1}}.
%Let $\bar v$ be the solution defined in \eqref{v}, and suppose there exists a different solution $\bar{\bar{v}}$ to the
%problem \eqref{problema_semicilindro} asymptotic to $\e^{\sqrt{\lambda_1} x}\psi_1(y)$ uniformly w.r.t. $y$.
%Let us apply the Phragm\`{e}n-Lindel\"{o}f Principle to the functions
%$w(x,y):=\bar v(x,y)-\bar{\bar v}(x,y)$ and $-w$, which are harmonic on $C^R$. Let us choose as comparison functions
%$w_k(x,y):=\bar v(x,y)+C$ for all $k$, being $C$ any positive constant,
%and $C^R_k:=\{(x,y)\in C^R:\ 0<x<k\}$.
%Therefore, Theorem \eqref{PL-principle} applies and
%$$ \liminf_{x\rightarrow+\infty} \sup_{(x,y)\in\Gamma'_k} \dfrac{\pm w(x,y)}{\bar v(x,y)+C} = 0 $$
%since $\pm w(x,y)=o(\e^{\sqrt{\lambda_1} x})$ whereas the
%denominator is asymptotic to $\e^{\sqrt{\lambda_1} x}$ as
%$x\rightarrow+\infty$ uniformly w.r. to $y$. This leads to
%$w\equiv0$. \qed
%
%\begin{prop}\label{step2}
%Any positive solution $v$ to the problem \eqref{problema_semicilindro} is $O(\e^{\sqrt{\lambda_1} x})$ as $x\rightarrow\infty$
%uniformly with respect to $y$.
%\end{prop}
%Given $v$ any positive solution to \eqref{problema_semicilindro}, let us denote
%\begin{eqnarray}
%R_n &:=& C^R_{n+1} \setminus C^R_n; \\
%b_n &:=& \max_{\overline{R}_n}\dfrac{v(x,y)}{\e^{\sqrt{\lambda_1} x}\psi_1(y)};\\
%a_n &:=& \min_{\overline{R}_n}\dfrac{v(x,y)}{\e^{\sqrt{\lambda_1} x}\psi_1(y)}.
%\end{eqnarray}
\begin{lem}[Boundary Harnack inequality, \cite{JK}]\label{boundary_Harnack}
Let $D\subset\R^N$, $N\geq2$, be a Lipschitz domain and let $V$ an
open set such that $V\cap\partial D\neq\emptyset$. Suppose $W$ is a
domain such that $W\subset D$, $\overline{W}\subset V$ and let $P_0$
be a point in $W$. Then there is a constant $C>0$ such that if $u$
and $v$ are nonnegative harmonic functions in $D$ which vanish on
$V\cap\partial D$ and satisfy $u(P_0)\leq v(P_0)$ then $u(P)\leq C
v(P)$ for all $P\in W$.
\end{lem}

Thanks to these two preliminary results, we can state
\begin{prop}\label{passo1}
Let $u$ and $v$ be two different positive solutions to the problem
\eqref{problema_semicilindro}. Then $u=O(v)$.
\end{prop}
\begin{proof}
According to the notation in Theorem \eqref{PL-principle}, let $D_k$
denotes the rectangle $\{(x,y)\in C^R,\ k-1<x<k \ \text{and}\ y\in
U^R\}$ and $\Gamma_k':=\{(k-1,y),\ y\in U^R\} \cup \{(k,y),\ y\in
U^R\}$. We can claim that
\begin{equation}\label{PL_pos}
\liminf_{k\rightarrow +\infty} \sup_{\Gamma_k'}
\dfrac{u(x,y)}{v(x,y)} >0.
\end{equation}
If not, Theorem
\eqref{PL-principle} would apply with $w_k=v\chi_{D_k}+\eps$ where
$\eps$ is any positive constant. Thus, we would obtain $u\leq 0$ and
then $u=0$, a contradiction.

We define $$ b_k = \max_{\Gamma_k'}\dfrac{u(x,y)}{v(x,y)} \qquad a_k
= \min_{\Gamma_k'}\dfrac{u(x,y)}{v(x,y)}. $$ Then, Equation
\eqref{PL_pos} implies $b_k \geq C>0$ for $k$ large enough and then,
by Lemma \eqref{boundary_Harnack} $\frac{u(x,y)}{v(x,y)} \leq C$ in
$C^R$. We can rewrite the previous inequality as
$\abs{\frac{u(x_1,y_1)}{v(x_1,y_1)}- \frac{u(x_2,y_2)}{v(x_2,y_2)}}
\leq C \frac{u(x_2,y_2)}{v(x_2,y_2)}$ for any $(x_1,y_1),\ (x_2,y_2)
\in C^R$ in order to obtain
\begin{equation}\label{ratio_bk_ak}
1 \leq \dfrac{b_k}{a_k} \leq C.
\end{equation}
This means that the two sequences $a_k$ and $b_k$ share the same
asymptotic behavior. Moreover, their divergence to $\infty$ cannot
occur. If they diverged to $+\infty$, then the inverse quotient
$\frac{v}{u}$ would be uniformly convergent to zero as $x\rightarrow
+\infty$, Theorem \eqref{PL-principle} would apply and provide the
contradiction $v\leq 0$.

In this way, they both cannot converge to zero, otherwise
\eqref{PL_pos} would be violated, and then
\begin{equation}\label{ratio_u_v}
C_1 \leq \dfrac{u}{v} \leq C_2
\end{equation}
for some positive constants $C_1$ and $C_2$.
\end{proof}

%\textit{Proof of Proposition (\ref{step2})}. We can rewrite the last
%inequality in the previous lemma as
%$\abs{\frac{u(P)}{v(P)}-\frac{u(Q)}{v(Q)}}\leq C \frac{u(Q)}{v(Q)}$
%for all $P,Q\in W$, eventually for a different $C$. The two
%functions $v$ and $\e^{\sqrt{\lambda_1} x}\psi_1(y)$ fulfill the
%hypothesis of the previous lemma, so that we have $b_n/a_n\leq C$,
%where $C$ is a fixed positive constant.
%
%It follows that the two sequences $a_n$ and $b_n$ share the same asymptotic behavior.
%If they diverge to $+\infty$, then the quotient $\dfrac{\e^{\sqrt{\lambda_1} x}\psi_1(y)}{v(x,y)}$ is uniformly
%convergent to zero as $x\rightarrow\infty$. We can again apply the Phragm\`{e}n-Lindel\"{o}f Principle
%and obtain the contradiction $\e^{\sqrt{\lambda_1} x}\psi_1(y) \leq 0$. \qed

\begin{prop}\label{step3}
 Any solution to \eqref{problema_semicilindro} is asymptotic to
$\e^{\sqrt{\lambda_k} x}\psi_k(y)$ as $x\rightarrow\infty$ uniformly with respect to $y\in U^R$ for some $k\in\N$,
where $\psi_k$ denotes one of eigenfunctions relative to the $k$-th eigenvalue of the problem 
\begin{equation}\label{k-esima_autofunzione}
\left\{ \begin{array}{ll} \Delta \psi_k=\lambda_k\psi_k &\text{in $U^R$} \\
\psi_k=0 &\text{on $\partial U^R$.}
\end{array}\right.
\end{equation}
\end{prop}
To prove this last step we need several preliminary results, which
are stated in Lemma \eqref{lemma_N_monotona}, Lemma
\eqref{lemma_N_costante} and Lemma \eqref{lemma_convergenza_v_xi}.

Being $v$ any solution to \eqref{problema_semicilindro}, we recall the Almgren frequancy function:
\begin{equation}\label{enne}
 N(v)(x):= \dfrac{\int_{\Omega_x}\abs{\nabla v}^2}{\int_{\Gamma_x}v^2},
\end{equation}
where $\Omega_x:=\{(\xi,\eta)\in \Omega :\ 0<\xi<x \}$ and $\Gamma_x:=\{(x,y):\ y\in U^R\}$.

\begin{lem}\label{lemma_N_monotona}
 Given a solution $v$ to \eqref{problema_semicilindro}, the function $N(v)(x) -\frac{C}{\sqrt{\lambda_1}}\e^{-\sqrt{\lambda_1}x}$
is monotone increasing with respect to $x$.
\end{lem}
\begin{proof}
It is simple to see that
\begin{eqnarray*}
 D'(x)=\int_{\Gamma_x}\abs{\nabla v}^2\\
 H'(x)=\int_{\Gamma_x}2\,v\,v_x.
\end{eqnarray*}
Multiplying the equation by $v_x$ and integrating by parts we obtain
$$ \int_{\Omega_x}\nabla v\,\nabla v_x = \int_{\Gamma_x}{v_x}^2 - \int_{\Gamma_0}{v_x}^2; $$
whereas differentiating it and multiplying it by $v$ we obtain
$$ \int_{\Omega_x}\nabla v\,\nabla v_x = \int_{\Gamma_x}v\,v_{xx} = -\int_{\Gamma_x}v\,\Delta_y v = \int_{\Gamma_x}{v_y}^2; $$
from which
$$ \int_{\Gamma_x}{v_x}^2 = \int_{\Gamma_x}{v_y}^2 + \int_{\Gamma_0}{v_x}^2. $$
 Let us compute the derivative
\begin{eqnarray*}
 \dfrac{d}{dx}N(x) &=& \dfrac{\int_{\Gamma_x}{v_x}^2+{v_y}^2}{\int_{\Gamma_x}v^2}
- 2\, \dfrac{\left(\int_{\Gamma_x}v\,v_x\right)^2}{\left(\int_{\Gamma_x}v^2\right)^2} \\
&=& \dfrac{2\int_{\Gamma_x}{v_x}^2 - \int_{\Gamma_0}{v_x}^2}{\int_{\Gamma_x}v^2}
- 2\, \dfrac{\left(\int_{\Gamma_x}v\,v_x\right)^2}{\left(\int_{\Gamma_x}v^2\right)^2}\\
&\geq& -\dfrac{\int_{\Gamma_0}{v_x}^2}{\int_{\Gamma_x}v^2}\\
&\geq& -\dfrac{C}{\e^{\sqrt{\lambda_1}x}}
\end{eqnarray*}
for some positive $C$: the first inequality is given by the H\"{o}lder inequality and the second one is implied by Proposition \eqref{passo1}.
\end{proof}

\begin{rem}\label{limitatezza_N}
Under our hypothesis we can claim $N(v)(x)$ admits a finite limit as
$x\rightarrow\infty$. Indeed, it admits a limit in view of Lemma \eqref{lemma_N_monotona}, and
 such a limit is finite since $v$ is $O(\e^{\sqrt{\lambda_1}x}\psi_1(y))$ from Proposition \eqref{passo1},
so that $N(v)$ is a bounded function from above.
%leaves
% $$\limsup_{k\rightarrow\infty}a_k=A \leq \liminf_{k\rightarrow\infty}b_k=B <\infty$$
% as the only possibility for these two limits. To see it, it is
% sufficient choosing $\e^{\sqrt{\lambda_1}x}\psi_1(y)$ as
% comparison function in the Theorem \eqref{PL-principle}.
\end{rem}

In order to detect $\lim_{x\rightarrow+\infty} N(v)(x)$ we introduce the sequence of normalized functions
$$v_\xi(x,y) := \dfrac{v(x+\xi,y)}{\left(\int_{\Gamma_\xi} v^2(\xi,y)\right)^{1/2}}
\qquad \textrm{for $\xi\in\R$, $x\in(0,1)$, $y\in U^R$}.$$

\begin{lem}\label{lemma_convergenza_v_xi}
 As $\xi\rightarrow\infty$ the sequence $\{v_\xi\}_{\xi}$ converges $\mathcal C^1$-uniformly on compact sets of the
cylinder $\{(x,y)\in\R^N:\ x\in\R\ \textrm{and}\ y\in \overline
U^R\}$ to a function harmonic on the cylinder whose $N(x)$ is
identically constant.
\end{lem}
\begin{proof}
 First we observe $N(v_\xi)(x)=N(v)(x+\xi)\leq \overline{N}$ for all $x\in(0,1)$ and for all $\xi\in\R$,
thanks to the definition of $v_\xi$ and to Remark
(\ref{limitatezza_N}). Thus, $\int_{\Omega_x}\abs{\nabla v_\xi}^2
\leq \overline{N}\int_{\Gamma_x}{v_\xi}^2$ where we recall
$$\int_{\Gamma_x}{v_\xi}^2 = \dfrac{\int_{U^R}v^2(x+\xi,y)\,dy}{\int_{U^R}v^2(\xi,y)\,dy}.$$
Via Harnack inequality, if $x$ ranges in a compact set, the previous
ratio is bounded from above by a fixed constant, then also the
$H^1$-norm is uniformly bounded from above. Thus, there exists a
subsequence at least $\mathcal C^1$-uniformly convergent to a
function $w$ which is harmonic on the whole cylinder. It holds for any
fixed $x\in\R$ $N(v_\xi)(x)=N(v)(x+\xi)\rightarrow \overline{N}$ as
$\xi\rightarrow\infty$, and then
$$\lim_{\xi\rightarrow\infty}N(v_\xi)(x)=\overline{N} \qquad \forall\ x\in\R.$$
Moreover this happens for any convergent subsequence. Then we can
conclude the whole sequence $v_\xi$ is $\mathcal C^1$-uniformly
convergent to a function $w$ which is harmonic on the whole cylinder and
has $N(x)$ identically constant.
\end{proof}

\begin{lem}\label{lemma_N_costante}
 Let $w$ be a solution to
\begin{eqnarray*}
\left\{ \begin{array}{ll}
 \Delta w=0 &\text{on $\{(x,y)\in\R^N:\ x\in\R\ \text{and}\ y\in U^R\}$}\\
w=0 &\text{if $y\in\partial U^R$}
\end{array} \right.
\end{eqnarray*}
with $\int_{y\in\overline U^R \atop x\leq \overline{x}}\abs{\nabla w}^2 < \infty$ for all $\overline{x}$.
Then $N(w)(x)$ is identically constant in $x$ if and only if $w(x,y)=\e^{\sqrt{\lambda_k} x}\psi_k(y)$ for some $k\in\N$,
being $\lambda_k$ the $k$-th eigenvalue of problem \eqref{k-esima_autofunzione} and $\psi_k$ one of its relative eigenfunctions.
\end{lem}
\begin{proof}
 Note for such solutions it holds $\int_{\Gamma_x}{w_x}^2=\int_{\Gamma_x}{w_y}^2$, so that
$$\dfrac{d}{dx}N(w)(x) = 2\, \dfrac{\int_{\Gamma_x}{w_x}^2}{\int_{\Gamma_x}w^2}
\left\{1 - \dfrac{\left(\int_{\Gamma_x}w\,w_x\right)^2}{\nor{w}_{L^2(\Gamma_x)}^2\nor{w_x}_{L^2(\Gamma_x)}^2} \right\}.$$
Thus, $N$ is identically constant in $x$ if and only if we have an equality in the H\"{o}lder inequality, that is
$$\left(\int_{\Gamma_x}w\,w_x\right)^2 = \int_{\Gamma_x}w^2\,\int_{\Gamma_x}{w_x}^2.$$
This happens if and only if $w_x(x,y)=\lambda(x)w(0,y)$, which leads to
$$w(x,y)=w(0,y)\left\{1 + \int_0^x \lambda(t)\,dt\right\}.$$
If we substitute this expression in $N(w)(x)\equiv N$ we obtain
$$\lambda(x) = N \left\{1 + \int_0^x \lambda(t)\,dt\right\}$$
which is a differential equation whose solution is $\lambda(x)=N\e^{Nx}$; then $w(x,y)=\e^{Nx}w(0,y)$, from which
$w(x,y)=\e^{\sqrt{\lambda_k} x}\psi_k(y)$ imposing $w$ is harmonic and zero on the boundary.
\end{proof}

\textit{Proof of Proposition \ref{step3}}.
We exploit the following chain of equalities:
$$ \lim_{x\rightarrow+\infty}N(v)(x) = \lim_{\xi\rightarrow+\infty}N(v)(x+\xi) =
\lim_{\xi\rightarrow+\infty}N(v_{\xi})(x) = N(w)(x) \equiv \sqrt{\lambda_k}. $$
Therefore Lemma \ref{lemma_N_costante} gives immediately the proof. \qed

\smallskip

%\begin{cor}
% The solution $\bar v$ defined in \eqref{v} is the unique solution to the problem \eqref{problema_semicilindro}
%up to multiplication by constants.
%\end{cor}
%\begin{proof}
\textit{Proof of Proposition \ref{step1}}. By Remark
\ref{limitatezza_N}
%$$\limsup_{n\rightarrow\infty}a_n=A \leq \liminf_{n\rightarrow\infty}b_n=B <\infty$$
%as the only possibility for these two limits.
we need to prove $A=B$.
%We invoke Proposition \ref{step3} together
%with Proposition \ref{step2} which forces $\lambda_k=\lambda_1$.
This is a straightforward consequence of Proposition \ref{step3}
where positivity of solutions forces $\lambda_k=\lambda_1$. \qed
%\end{proof}

\subsection{The case $U^R$ hyperplane.}

%In this case if we do not suppose the energy of the solution is finite, in general solutions are not unique,
%since every linear function which is zero on the boundary is a solution to the problem \eqref{problema_semicilindro}.
%Rather, if we are looking for solutions which are in fact unique, we need to assume their energy
%$\int_{C^R}\abs{\nabla v}^2$ is finite.
%
%Further, we point out that in dimension $N=2$ under this hypothesis the problem \eqref{problema_semicilindro}
%is essentially equivalent to the same problem on the ball, and then it admits no non-trivial solution.
%Then, for this particular dimension, we will need to assume non-trivial boundary conditions.

The existence of a positive solution in this case is immediately proved by considering the function $\bar v(x,y):= x$,
where we recall $x$ denotes the first variable in $\R^N$.

We aim to prove this is in fact the unique solution to the problem
\eqref{problema_semicilindro} when $U^R$ is a whole hyperplane of
$\R^{N}$, namely $\{x=0\}$. To do this, we follow the same outline
as before.

Let $B_r$ be the ball in $\R^N$ centered in the origin with radius $r$, we denote
$$C_r := C^R \cap B_r \qquad \text{and} \qquad \Gamma_r:= \partial B_r \cap C^R.$$

%\begin{prop}\label{step1hyperplane}
%The function $v$ defined in \eqref{x} is the unique solution with that particular asymptotic behavior.
%\end{prop}
%
\begin{prop}\label{step2hyperplane}
Any positive solution to the problem \eqref{problema_semicilindro}
is $O(x)$ as $x\rightarrow\infty$ uniformly with respect to $y$.
\end{prop}
The proof of this proposition is essentially the same as in the
previous case, provided the domains $D_k$ are now defined as
$C_k$.

\begin{prop}\label{step3hyperplane}
 Any solution to \eqref{problema_semicilindro} is asymptotic to
$r^{\overline N}v(1,\theta)$ as $r\rightarrow\infty$ uniformly with respect to $\theta\in\sphere{N-1}$ in such a way that
$[\overline N(\overline N -1) + \overline N (N-1)]$ is an eigenvalue for the spherical Laplacian
and $v(1,\theta)$ is one of its relative eigenfunctions.
\end{prop}

To prove this last step we need several preliminary results, which we state in the Lemma \eqref{lemma_N_monotona_hyperplane},
Lemma \eqref{lemma_N_costante_hyperplane} and Lemma \eqref{lemma_convergenza_v_r_hyperplane}.

We aim to pursue again an Almgren-type argument on the domains
$C_r$. Being $v$ any solution to \eqref{problema_semicilindro}, let
us introduce the following Almgren-type quotient
\begin{equation}\label{enne_hyperplane}
 N(v)(r):= \dfrac{r^{2-N}\int_{C_r}\abs{\nabla v}^2}{r^{1-N}\int_{\Gamma_r}v^2}=: \dfrac{D(r)}{H(r)}.
\end{equation}

\begin{lem}\label{lemma_N_monotona_hyperplane}
 Given a solution $v$ to \eqref{problema_semicilindro}, the quotient $N(v)(r)$ is monotone increasing with respect to $r$.
\end{lem}
\begin{proof}
It is quite simple to see
\begin{equation}\label{H'(r)}
H'(r) = 2 r^{1-N} \int_{\Gamma_r}v\,v_r.
\end{equation}
Testing the equation by $v$ we obtain
\begin{equation}
H'(r) = 2 r^{1-N} \int_{C_r}\abs{\nabla v}^2 = \dfrac{2}{r} D(r);
\end{equation}
from which $D(r)=(r/2) H'(r)$.

On the other hand we claim
\begin{equation}\label{D'(r)}
D'(r)=2 r^{2-N} \int_{\Gamma_r}{v_r}^2.
\end{equation}
Indeed,
\begin{equation}\label{D'(r)1}
D'(r)=(2-N)r^{1-N}\int_{C_r}\abs{\nabla v}^2 + r^{2-N} \int_{\Gamma_r}\abs{\nabla v}^2;
\end{equation}
testing the equation with $\nabla v \cdot (x,y)$ and integrating by parts we obtain
\begin{equation}\label{D'(r)2}
\int_{C_r}\nabla v \cdot \nabla (\nabla v\cdot (x,y)) = r \int_{\Gamma_r}{v_r}^2,
\end{equation}
which is in fact
\begin{equation}\label{D'(r)3}
\int_{C_r}\nabla v \cdot \nabla (\nabla v\cdot (x,y)) = -\dfrac{N-2}{2}\int_{C_r}\abs{\nabla v}^2
+ \dfrac{r}{2} \int_{\Gamma_r}\abs{\nabla v}^2
\end{equation}
via integration by parts.
From \eqref{D'(r)1}, \eqref{D'(r)2} and \eqref{D'(r)3} we immediately obtain \eqref{D'(r)}.

Now, the derivative of $N$ is of course
$N'(r)=\dfrac{D'(r)H(r)-D(r)H'(r)}{H^2(r)}$, and we recall that $D(r)H'(r)=(r/2)(H'(r))^2$, so that
\begin{eqnarray*}
N'(r)
%&=& \dfrac{\displaystyle 2 r^{2-N}r^{1-N}\int_{\Gamma_r}{v_r}^2\int_{\Gamma_r}{v}^2
%- 2 r^{3-2N}\left(\int_{\Gamma_r}{v_r}v\right)^2}{H^2(r)} \\
= \dfrac{2 r^{3-2N}}{H^2(r)}\left\{\int_{\Gamma_r}{v_r}^2\int_{\Gamma_r}{v}^2
- \left(\int_{\Gamma_r}{v_r}v\right)^2\right\} \geq 0
\end{eqnarray*}
thanks to the H\"{o}lder inequality.
\end{proof}

Now we introduce the sequence of normalized functions
$$v_r(x,y) := \dfrac{v(rx,ry)}{\left(\int_{\Gamma_{1/2}} v^2(rx,ry)\right)^{1/2}} \qquad \textrm{for $r>1$}.$$

\begin{lem}\label{lemma_convergenza_v_r_hyperplane}
 As $r\rightarrow\infty$ the sequence $\{v_r\}_{r}$ converges $\mathcal C^1$-uniformly on $C_1$
to a function which is harmonic on the whole halfspace and whose $N(x)$ is identically constant.
\end{lem}
\begin{proof}
Here the proof is essentially the same as in Lemma \eqref{lemma_convergenza_v_xi}.
\end{proof}

\begin{lem}\label{lemma_N_costante_hyperplane}
Let $v$ any non-trivial solution to the problem \eqref{problema_semicilindro}.
Then its Almgren's frequency function is identically constant equal to $\overline N$ if and only if
$$v(r,\theta) = r^{\overline N} v(1,\theta)$$
in such a way that
$[\overline N(\overline N -1) + \overline N (N-1)]$ is an eigenvalue for the spherical Laplacian
and $v(1,\theta)$ is one of its relative eigenfunctions.
\end{lem}
\begin{proof}
If the derivative of the frequency function is identically zero, then an equality must hold
in the H\"{o}lder inequality, so that $v_r(r,\theta)=\lambda(r)v(1,\theta)$, that is
$$v(r,\theta)= v(1,\theta)\{1+\int_1^r \lambda(t)dt\}.$$
Imposing $D(r)/H(r)=(r/2)(H'(r)/H(r)) = \overline N$ we obtain
\begin{eqnarray*}
\overline N &=& \dfrac{\displaystyle r\int_{\Gamma_r}v\,v_r}{\displaystyle \int_{\Gamma_r}v^2}
= r \dfrac{\displaystyle \int_{\Gamma_r}v^2(1,\theta) \lambda(r) \left(1+\int_1^r \lambda(t)dt\right)d\theta}
{\displaystyle \int_{\Gamma_r}v^2(1,\theta)\left(1+\int_1^r \lambda(t)dt\right)^2\,d\theta}
= \dfrac{r\,\lambda(r)}{\displaystyle 1+\int_1^r \lambda(t)dt}.
\end{eqnarray*}
The solution of the ordinary differential equation
$$r\,\lambda(r) = \overline N \left\{1+\int_1^r \lambda(t)dt\right\}$$
is indeed $\int_1^r \lambda(t)dt = r^{\overline N}-1$, which leads to
$v(r,\theta)=r^{\overline N} v(1,\theta)$.
Imposing $v$ is harmonic on the whole halfspace, we deduce the conditions on $\overline N$ and $v(1,\theta)$.
\end{proof}

%\begin{prop}\label{step3hyperplane}
%Given any positive solution $v$, the sequence
%\begin{equation}\label{v_r}
%v_r(x,y) = \dfrac{v(rx,ry)}{\left(\int_{\Gamma_{1/2}}v^2(rx,ry)\right)^{1/2}}
%\end{equation}
%converges uniformly w.r. to $y$ to a function $v$ which is harmonic on the whole
%halfspace and whose Almgren's frequency function is constant.
%\end{prop}
%
%\begin{prop}\label{step4hyperplane}
%Let $v$ any non-trivial solution to the problem \eqref{problema_semicilindro}.
%Then its Almgren's frequency function is identically constant equal to $\overline N$ if and only if
%$$v(r,\theta) = r^{\overline N} v(1,\theta)$$
%in such a way that
%$[\overline N(\overline N -1) + \overline N (N-1)]$ is an eigenvalue for the spherical Laplacian
%and $v(1,\theta)$ is its relative eigenfunction.
%\end{prop}

\begin{cor}
 The solution $\bar v$ defined in \eqref{x} is the unique positive solution to the problem \eqref{problema_semicilindro}
up to multiplication by constants.
\end{cor}
\begin{proof}
Positivity assumption forces $\overline N = 1$ in Proposition \eqref{step3hyperplane}.
This homogeneity degree together with $v(0,y)=0$ implies $v(x,y)=x$.
\end{proof}

\section{Solutions on $\Omega$}

\subsection{%Existence and uniqueness of a
Positive solutions on $\Omega$ with finite energy on $C^L$.}

The following proposition can be easily proved.

\begin{prop}
Let us consider the case $\Omega:=C^L \cup C^R$ where $U^R$ is the
hyperplane $\{x=0\}$.  Let $\Phi$ be unique normalized positive solution of \eqref{problema_semicilindro}, extended as vanishing outside the semicylinder.  There exists a unique positive solution $v$ to
problem \eqref{problema_omega} such that $u=v-\Phi$ has finite energy on $\Omega$: it is the solution of the minimum
problem
\begin{equation}\label{min_prob}
\min_{u\in \mathcal D^{1,2}(\Omega)} \dfrac{1}{2}\int_{\Omega}
\abs{\nabla u}^2 - \int_{\Gamma}\dfrac{\partial \Phi}{\partial
x}_{|x=0} u.
\end{equation}
\end{prop}
%\begin{proof}
%
%The existence of a minimum is obvious.
%First we observe that the functional is $\mathcal C^1$, then the
%minimum is also a critical point. Furthermore, thanks to term
%$\int_{\Gamma}u$, the trivial function does not solve the
%Euler-Lagrange equation, so that the sought minimum is not trivial
%for sure. Let us consider a minimizing sequence $\{u_n\}$.
%Restricting our analysis to positive solutions \todo{da capire bene:
%si puo' fare??}, we may assume $\{u_n\}$ is equibounded in
%$L^1(\Gamma)$ norm. Since it is minimizing, $\{u_n\}$ is equibounded
%in $\mathcal D^{1,2}(\Omega)$, convergent to $\overline u$ weakly in
%$\mathcal D^{1,2}(\Omega)$ and strongly in $L^1(\Gamma)$; so that
%the limit $\overline u$ is not trivial. Moreover, the term
%$\int_{\Omega}\abs{\nabla u}^2$ is weakly lower semicontinuos, so
%that the infimum is achieved.
%\end{proof}

We note that the minimizer  $u$  is not a $\mathcal C^1$
solution. Indeed, on one hand for every $\varphi\in\mathcal
D^{1,2}(\Omega)$ we have
\begin{equation}\label{eq_debole_probl_min}
\int_{\Omega}\nabla u\nabla \varphi = \int_{\Gamma}\dfrac{\partial
\Phi}{\partial x}_{|x=0}\varphi;
\end{equation}
whereas on the other hand, multiplying the equation by $\varphi$ nd
integrating by parts over $C^L$ and $C^R$ we obtain
$$\int_{\Omega}\nabla u\nabla \varphi = \int_{C^L\cup C^R}\nabla u\nabla \varphi
= \int_{\Gamma} \varphi \left(-\dfrac{\partial u^R}{\partial
x}_{|x=0} + \dfrac{\partial u^L}{\partial x}_{|x=0} \right)$$ where
$u^L:= u \chi_{C^L}$ and $u^R$ is defined similarly. Thus,
\begin{equation}\label{saltoderivata}
\dfrac{\partial u^L}{\partial x}_{|x=0} = \dfrac{\partial
u^R}{\partial x}_{|x=0} +\dfrac{\partial \Phi}{\partial x}_{|x=0},
\end{equation}
in the sense that must be specified yet (see Section 4). In order to abtain a
 $\mathcal C^1$ solution, we need to consider the
sum $v=u+\Phi$ instead of $u$.

Furthermore, if the test function $\varphi$ has compact support far
away from $\Gamma$, Equation \eqref{eq_debole_probl_min} shows that
the minimum is a harmonic function in
$\Omega\setminus\Gamma$. In this way, if we are looking for a
harmonic function $u+\Phi$ on the whole $\Omega$, $\Phi$ must be the
unique (up to multiplication by constants) solution to the problem
\eqref{problema_semicilindro} (see the previous section). In other words,
given the function $\Phi$ solution to the problem \eqref{problema_semicilindro},
the function $u+\Phi$ is the unique solution to the problem \eqref{problema_omega}
with finite energy on the left. Furthermore, it is possible to prove that any
positive solution to the problem \eqref{problema_omega} with finite energy on the left
takes the form $u+\Phi$ for a certain $\Phi$ solution to the problem \eqref{problema_semicilindro},
in order to state the following

\begin{thm}\label{teorema_omega_energia_finita_da_una_parte}
There exists a unique (up to multiplicative constants) solution to the problem
\eqref{problema_omega} having finite energy on $C^L$ and satisfying
\[    \lim_{x\rightarrow +\infty} N(x) = \sqrt{\lambda_1^R}.
\]
It is asymptotic to a multiple of \eqref{v} if $U^R$ is bounded,
whereas it is asymptotic to a multiple of \eqref{x} if $U^R$ is a
whole hyperplane.
\end{thm}
\begin{proof}
 The proof follows the same outline as the proof of Theorem \eqref{teorema_semicilindro}.

Propositions \eqref{passo1} and \eqref{step2hyperplane} can be stated and proved in the same way
choosing $D_k=\{(x,y)\in \Omega,\ -k < x < k\}$ in the first case and
$D_k=\{(x,y)\in \Omega,\ -k < x \leq 0\} \cup C_k$ in the second case.

We conclude the proof throughout an Almgren type argument on the domains
$\Omega_x = \{ (\xi,\eta)\in\Omega:\ \xi<x \}$ (but now $\Gamma_0 = \{ x=0 \} \cap \partial\Omega$)
in the first case and $\Omega_r = \{ (x,y)\in\Omega: x\leq0 \} \cup C_r$ in the second case.
In both cases the computations are the same.
\end{proof}

\begin{rem}\label{remark_compliance}
 As already highlighted in \cite{AFT13}, the minimum in Equation \eqref{min_prob} is strictly related to
the concept of \emph{compliance}. We define
\begin{equation}\label{def_compliance}
  \mathfrak{C}(\Gamma):=
\max_{w\in \mathcal
  D^{1,2}(\Omega)} \bigg(2
\int_{\Gamma}\dfrac{\partial \Phi}{\partial
x}_{|x=0} w-\int_{\Omega}|\nabla w|^2\,dx\bigg)
\end{equation}
the compliance functional
associated to a force concentrated on the section
$\Gamma$ in the flavor of \cite{BS,BSV}.  In general, the compliance
functional measures the rigidity of a membrane subject to a given
(vertical) force: the maximal rigidity is obtained by minimizing the
compliance functional $\mathfrak C(\Gamma)$ in a certain class of admissible
regions $\Gamma$.
\end{rem}

\subsection{Infinite energy solutions.}

Up to now, we have proved that given a positive profile $\phi$ on
$U^R$, there exist at least two positive solutions to the problem
\begin{eqnarray}\label{problema_semicilindro_profilo_positivo}
\left\{
  \begin{array}{ll}
    \Delta w =0, & \text{in $C^R$;} \\
    w = \phi, & \text{on $U^R$;} \\
    w = 0, & \text{on $\partial C^R \setminus U^R$.}
  \end{array}
\right.
\end{eqnarray}
Indeed, one has finite energy and it is the minimum of the Dirichlet
realization on $C^R$, we name it $u$; whereas the second one is
obtained from the previous simply adding a multiple of the solution
$v$ of the Theorem \eqref{teorema_semicilindro}.

\begin{thm}
Any positive solution to the problem
\eqref{problema_semicilindro_profilo_positivo} is a linear
combination $ u + c v$ with $c \geq 0$, being $u$ and $v$ as
mentioned above.
\end{thm}
\begin{proof}
Let $w>0$ be a solution to the problem
\eqref{problema_semicilindro_profilo_positivo}. If its energy is
finite, then it coincides with $u$ since in this case we have
uniqueness of solution.

If $w$ has an infinite energy, consider the difference $w-u$. Then,
we can immediately state that
$$\liminf _{x\rightarrow +\infty} \sup_{\Gamma_x'} \dfrac{w-u}{v} > 0$$
since if not, the Phragm\'{e}n-Lindel\"{o}f Theorem would imply
$w-u\leq 0$, a contradiction. As in the proof of Proposition
\eqref{passo1} we obtain
\begin{equation}\label{stima_grezza_w}
c_1 \leq \dfrac{w-u}{v} \leq c_2.
\end{equation}
We follow the same
outline as before and study the Almgren quotient $N(x)$ on
$\Omega_x^0:=\{(\xi,\eta)\in \R^N:\ \xi\in (0,x),\ \eta\in U^R\}$.
As before, $N(x)=\frac{D(x)}{H(x)}$ where
$D(x)=\int_{\Omega_x^0}\abs{\nabla w}^2$ and
$H(x)=\int_{\Gamma_x}w^2$ being $\Gamma_x=\{(x,\eta):\ \eta \in
U^R\}$. Multiplying the Laplace equation by $w$ itself, we obtain
$$\int_{\Omega_x^0}\abs{\nabla w}^2 = \int_{\Gamma_x}w\,w_x - \int_{\Gamma_0}w\,w_x.$$
Multiplying the Laplace equation by $w_x$ we obtain
$$\int_{\Omega_x^0}\nabla w \cdot \nabla w_x = \int_{\Gamma_x}{w_x}^2 - \int_{\Gamma_0}{w_x}^2$$
where
$$\int_{\Omega_x^0}\nabla w \cdot \nabla w_x = \int_{\partial \Omega_x^0}\dfrac{1}{2}\abs{\nabla w}^2\,\nu\cdot e_1
= \int_{\Gamma_x}\dfrac{1}{2}\abs{\nabla w}^2 -
\int_{\Gamma_0}\dfrac{1}{2}\abs{\nabla w}^2$$ so that
$$\int_{\Gamma_x}\abs{\nabla w}^2 = \int_{\Gamma_0}\abs{\nabla w}^2 + 2\int_{\Gamma_x}{w_x}^2 - 2\int_{\Gamma_0}{w_x}^2.$$
Thus, the derivative
\begin{eqnarray*}
N'(x) &=& \dfrac{D'(x)H(x)-D(x)H'(x)}{H^2(x)} \\
&=& \dfrac{\displaystyle\left(\int_{\Gamma_0}\abs{\nabla w}^2 +
2\int_{\Gamma_x}{w_x}^2 - 2\int_{\Gamma_0}{w_x}^2
\right)\int_{\Gamma_x}w^2 - \left(\int_{\Gamma_x}{w_x}^2 -
\int_{\Gamma_0}{w_x}^2\right)
\int_{\Gamma_x}2w\,w_x}{\displaystyle\left(\int_{\Gamma_x}w^2\right)^2}\\
&=&
\dfrac{\displaystyle2\left\{\int_{\Gamma_x}{w_x}^2\int_{\Gamma_x}{w}^2
- \left(\int_{\Gamma_x}w\,w_x\right)^2 \right\} +
\int_{\Gamma_0}{w_y}^2\int_{\Gamma_x}{w}^2 -
\int_{\Gamma_0}{w_x}^2\int_{\Gamma_x}{w}^2 +
2\int_{\Gamma_0}w\,w_x\int_{\Gamma_x}w\,w_x
}{\displaystyle\left(\int_{\Gamma_x}w^2\right)^2}\\
&\geq&
\dfrac{\displaystyle\int_{\Gamma_0}{w_y}^2-{w_x}^2}{\displaystyle\int_{\Gamma_x}w^2}
+ \dfrac{\displaystyle
2\int_{\Gamma_0}w\,w_x\int_{\Gamma_x}w\,w_x}{\displaystyle\left(\int_{\Gamma_x}w^2\right)^2}
\end{eqnarray*}
via H\"{o}lder inequality. Thanks to the estimate
\eqref{stima_grezza_w} the function
$$\dfrac{\displaystyle\int_{\Gamma_0}{w_y}^2-{w_x}^2}{\displaystyle\int_{\Gamma_x}w^2}
+ \dfrac{\displaystyle
2\int_{\Gamma_0}w\,w_x\int_{\Gamma_x}w\,w_x}{\displaystyle\left(\int_{\Gamma_x}w^2\right)^2}
\in L^1(0,+\infty),$$ so that $N(x)$ admits a limit as
$x\rightarrow+\infty$. Moreover, such a limit is finite since the
quantities $a_k$ and $b_k$ cannot diverge to infinity via Lemma
\eqref{boundary_Harnack} and Theorem \eqref{PL-principle} as in the
proof of Proposition \eqref{passo1}. We conclude the proof invoking
Proposition \eqref{step3}.
\end{proof}

\begin{thm}
Any positive solution to the problem \eqref{problema_omega} is a
linear combination $c^L v^L + c^R v^R$ with $c^L,\,c^R \geq 0$ (at least one of the two constants must be
different from zero), where
$v^L$ and $v^R$ are the solutions in the Theorem
\eqref{teorema_omega_energia_finita_da_una_parte} with finite energy
on $C^L$ and $C^R$ respectively.
\end{thm}
\begin{proof}
The proof relies essentially on the Phragmen-Lindel\"{o}f Principle.
Let $w>0$ a solution to the problem \eqref{problema_omega}. We
simply apply the aforementioned principle on $w-(c^L v^L + c^R v^R)$
comparing it with $c^L v^L + c^R v^R +1$. In this case we choose the
sequence of domains $D_k$ as the union $\{(\xi,\eta): \xi\in(0,k)\
\eta \in U^R\} \cup \{(\xi,\eta): \xi\in(-k,0)\ \eta \in U^L\}$
whenever $U^R$ is bounded, whereas $\{(\xi,\eta): \xi\in(-k,0)\ \eta
\in U^L\} \cup (C^R \cap B_k(\overline x))$ where $\overline x$ is
the junction point between $C^L$ and $C^R$ whenever $U^R$ is the
whole hyperplane.
\end{proof}
We stress that such solutions have
$\lim_{x\rightarrow\pm\infty}N(x)$ lowest as possible in order to be nontrivial, that is
$\sqrt{\lambda_1^R}$ and $\sqrt{\lambda_1^L}$ respectively. 
We note that $v^L$ is asymptotic to a multiple of $ \e^{-\sqrt{\lambda_1^L}x} \psi_1^L$
as $x\to-\infty$.
Does
the reverse implication hold true? Not exactly, but we can state
\begin{thm}\label{teo_sol_freq_lim}
The function set
\begin{eqnarray*}
\mathcal S:=\left\{w \text{ solution to \eqref{problema_omega} s.t.
} \lim_{x\rightarrow +\infty} N(x) \leq \sqrt{\lambda_1^R} \text{ or }
    \lim_{x\rightarrow -\infty} N(x) \leq \sqrt{\lambda_1^L} \right\}
\end{eqnarray*}
is a linear space of dimension 2 and $\{v^L,v^R\}$ is a basis, being
$v^L,\,v^R$ as in the previous theorem.
\end{thm}
We remark that in this case no positivity assumption can be made on
solutions, but we can state that they change their sign at most just
one time.

\begin{rem}\label{rem:sign_changing}
The procedure presented up to now works even in the case the upper bound for the Almgren frequency 
is set to be a $k$-th eigenvalue of the problem \eqref{k-esima_autofunzione} with $k\geq 2$, up to minor modifications.
This allows us to extend Theorem \ref{teorema_omega_energia_finita_da_una_parte} 
providing the following
 \begin{thm}
Let $\lambda_k^R$ be the $k$-th eigenvalue of the problem \eqref{k-esima_autofunzione} 
and let us denote $m_k^R$ its multiplicity. Then, there exist exactly 
$m_k^R$ linearly independent solutions to the problem \eqref{problema_omega} having finite energy 
on $C^L$ and satisfying
\[    \lim_{x\rightarrow +\infty} N(x) = \sqrt{\lambda_k^R}.
\] 
Each of them is asymptotic to a multiple of $\e^{\sqrt{\lambda_k^R}x}\psi_k^R(y)$, 
being $\psi_k^R$ one the eigenfunctions relative to $\lambda_k^R$, if $U^R$ is bounded, whereas 
each of them is asymptotic to a multiple of $r^{\overline N} v(1,\theta)$ if $U^R$ is a whole hyperplane, 
in such a way that $[\overline N(\overline N -1) + \overline N (N-1)]$ is an eigenvalue for the spherical Laplacian
and $v(1,\theta)$ is one of its relative eigenfunctions.
 \end{thm}
Thus, Theorem \ref{teo_sol_freq_lim_sign_changing} is finally proved.
\end{rem}

\section{Frequency transfer from two consecutive cylinders}\label{sec:transfer}

Let us focus our attention on the unique solution which has finite
energy at $-\infty$. We are talking about $u+\Phi$, where $u$ is the
minimum of \eqref{min_prob} and $\Phi$ the unique solution of the
problem \eqref{problema_semicilindro}. Thanks to the uniqueness of
such a solution, whenever we impose the exact behavior at
$x\rightarrow +\infty$, the asymptotic behavior for $x\rightarrow
-\infty$ is determined. We aim to investigate how such a fact
occurs.

\begin{rem}\label{andamento_u^L,R}
Via the Phragm\`{e}n-Lindel\"{o}f Theorem, the restrictions
$u^L:=u\chi_{C^L}$ and $u^R:=u\chi_{C^R}$ are $u^L(x,y) =
O(\e^{\sqrt{\lambda_1^L}x}\varphi_1^L(y))$ whereas
$u^R(x,y)=O(\e^{-\sqrt{\lambda_1^R}x}\varphi_1^R(y))$. Indeed, given
the particular domain's geometry, $u^L$ and $u^R$ can be written as
$\sum_k c_k^L(x)\varphi_k^L(y)$ and $\sum_k c_k^R(x)\varphi_k^R(y)$
respectively. Then, imposing that $\Delta u^i=0$ for $i=L,R$ and
that their energy is finite, they take the form
\begin{equation}\label{sviluppi}
u^L(x,y)=\sum_k \alpha_k\e^{\sqrt{\lambda_k^L}x}\varphi_k^L(y)
\qquad u^R(x,y)=\sum_k
\beta_k\e^{-\sqrt{\lambda_k^R}x}\varphi_k^R(y)
\end{equation}
where the eigenfunctions $\{\varphi_k^L\}$ and $\{\varphi_k^R\}$ are
basis for $L^2(U^L)$ and $L^2(U^R)$ respectively.
\end{rem}

The key points for this analysis are Equation \eqref{saltoderivata}
together with the fact that the two profiles of $u^L$ and $u^R$
coincides on the boundary $\{(x,y)\in\Omega,\, x=0\}$.

In particular, Equation \eqref{saltoderivata} makes sense in a
distributional sense, so that it should be read in the dual space
$H^{-1/2}(U^L)$. Indeed, both $u^L$ and $u^R$ are $\D^{1,2}$
functions on $C^L$ and $C^R$ respectively, then their traces on
$\{x=0\}$ are $H^{1/2}$ functions and then their partial derivatives
on $\{x=0\}$ are in $H^{-1/2}(U^L)$ and $H^{-1/2}(U^R)$
respectively. In order to specify these concepts, we introduce the
following spaces
\begin{eqnarray*}
\mathfrak{h}^{1/2}_L &:=& \{(\alpha_j)_j \ \text{s.t.}\ \sum_j
\left(\lambda_j^L\right)^{1/2}{\alpha_j}^2<+\infty \},\\
\mathfrak{h}^{1/2}_R &:=& \{(\alpha_j)_j \ \text{s.t.}\ \sum_j
\left(\lambda_j^R\right)^{1/2}{\alpha_j}^2<+\infty \},
\end{eqnarray*}
being $\lambda_j^L$ and $\lambda_j^R$ the eigenvalues of
$\Delta_{N-1}$ on $U^L$ and $U^R$ respectively, and operators
\begin{eqnarray*}
\U :  \mathfrak h^{1/2}_L &\longrightarrow& \mathfrak h^{1/2}_R \\
\alpha=(\alpha^j)_j &\longmapsto& (\U(\alpha))^k=\U^k_j \alpha^j\\
\widetilde{\U} : H^{1/2}(U^L) &\longrightarrow& H^{1/2}(U^R)\\
u=\alpha^j \varphi_j^L &\longmapsto&
\widetilde{\U}u=(\U^k_j\alpha^j) \varphi_k^R.
\end{eqnarray*}
Moreover, $\widetilde{\mathcal U}^*: H^{-1/2}(U^R) \longrightarrow
H^{-1/2}(U^L)$ will be the adjoint operator.

These mean that Equation \eqref{saltoderivata} is correctly read as
\begin{equation}
\dfrac{\partial u^L}{\partial x}_{|x=0} = \widetilde{\mathcal
U}^*\left(\dfrac{\partial u^R}{\partial x}_{|x=0}\right)
+\widetilde{\mathcal U}^*\left(\dfrac{\partial \Phi}{\partial
x}_{|x=0}\right) \qquad \text{in $H^{-1/2}(U^L)$}.
\end{equation}
which is
\begin{equation}
\alpha_j\sqrt{{\lambda_j}^L}{\varphi_j}^L - \widetilde{\mathcal
U}^*\left(\beta_k\sqrt{{\lambda_k}^R}{\varphi_k}^R\right) =
\widetilde{\mathcal U}^*\left(\gamma_k{\varphi_k}^R\right)
\end{equation}
where $\gamma_k$ are the coefficients of $\frac{\partial
\Phi}{\partial x}_{|x=0}$. Thus the equation for the coefficients
becomes
\begin{eqnarray}\label{saltoderivata_coefficienti_giusta}
\alpha_j\sqrt{{\lambda_j}^L} - \mathcal
U^*\left(\beta_k\sqrt{{\lambda_k}^R}\right) = \mathcal
U^*\left(\gamma_k\right) \nonumber \\
\alpha_j\sqrt{{\lambda_j}^L} - \mathcal
U^*\left(\sqrt{{\lambda_k}^R}\mathcal U^j_k \alpha_j\right) =
\mathcal U^*\left(\gamma_k\right)
\end{eqnarray}
since $\beta_k=\alpha_j \mathcal U^{jk}$ from the fact
$u^L(0,y)=u^R(0,y)=\sum_k\beta_k\psi_k^R(y)$.

Equation \eqref{saltoderivata_coefficienti_giusta} becomes
\begin{eqnarray}\label{saltoderivata_coefficienti_giusta_2}
\Lambda^L\alpha - \mathcal U^*\Lambda^R\mathcal U\, \alpha &=&
\alpha_0 \nonumber \\
\left(\Lambda^L - \mathcal U^*\Lambda^R\mathcal U\right)\alpha &=&
\alpha_0\nonumber \\
\left(\mathbb I - (\Lambda^L)^{-1}\mathcal U^*\Lambda^R\mathcal
U\right)\alpha &=& \alpha_0
\end{eqnarray}
where $\alpha_0=(\Lambda^L)^{-1}\mathcal U^*\left(\gamma_k\right)$,
$\Lambda^R$ the diagonal operator between $\mathfrak{h}^{1/2}_R$ and
$\mathfrak{h}^{-1/2}_R$ which multiplies by the square root of the
eigenvalues $\sqrt{{\lambda_j}^R}$, which is in fact an isometry
between those two spaces, whereas $(\Lambda^L)^{-1}$ is analogously
an isometry from $\mathfrak{h}^{-1/2}_L$ into
$\mathfrak{h}^{1/2}_L$.

\begin{prop}
The operator $T=(\Lambda^L)^{-1}\mathcal U^*\Lambda^R\mathcal U$ is
a contraction on $\mathfrak{h}^{1/2}_L$.
\end{prop}
\begin{proof}
Proving that $\mathcal U^*\Lambda^R\mathcal U$ has got the same
eigenvalues of $\Lambda^R$ will be sufficient to our aim. Once we
have that, we apply the well-known Weyl's law: being $\lambda_j$ the
$j$-th eigenvalue of the Laplacian on a bounded regular domain
$\Omega$ of dimension $n$, the following asymptotic behavior holds
$\lambda_j \sim C_n j^{2/n} \abs{\Omega}^{-2/n}$ as $j\rightarrow
+\infty$ and $C_n$ is a constant depending only on the dimension
$n$. Then, not only the ratio $\frac{\lambda_j^R}{\lambda_j^L}<1$
and then $T$ is a contraction at every point, but also the ratio is
uniformly far away from 1, so that $T$ is a contraction on the whole
space $\mathfrak{h}^{1/2}_L$.
%the fact ${\lambda_j}^L>{\lambda_j}^R$ for every $j$ implies the
%thesis.

Let us study the eigenvalues of $\mathcal U^*\Lambda^R\mathcal U$.
First of all we note that $\mathcal U$ is a bounded operator from
$\mathfrak{h}^{1/2}_L$ into $\mathfrak{h}^{1/2}_R$ with operator
norm less or equal to 1. In fact, $\widetilde{\mathcal U}$ is an
isometry
%($\nor{u}^2_{L^2(U^L)}=\nor{u}^2_{L^2(U^R)}$)
from $L^2(U^L)$ into $L^2(U^R)$ as well as from $H^1_0(U^L)$ into
$H^1_0(U^R)$. Being $H^{1/2}(U^L)$ and $H^{1/2}(U^R)$ intermediate
spaces $[L^2(U^L),H^1(U^L)]_{1/2}$ and $[L^2(U^R),H^1(U^R)]_{1/2}$
respectively, the operator $\widetilde{\mathcal U}: H^{1/2}(U^L)
\rightarrow H^{1/2}(U^R)$ has operator norm
$$\nor{\widetilde{\mathcal U}} \leq \nor{\widetilde{\mathcal U}}_{L^2(U^L),H^1(U^L)}
\cdot \nor{\widetilde{\mathcal U}}_{L^2(U^R),H^1(U^R)} \leq 1$$ (see
\cite{A}).

Secondly, $\mathfrak{h}^{1/2}_L \subset \mathfrak{h}^{1/2}_R$ thanks
to the relation between the eigenvalues mentioned above.

%implies
%\begin{equation}\label{norma_L2}
%\sum_j {\alpha_j}^2 = \sum_k \left(\mathcal U ^j_k
%\alpha_j\right)^2,
%\end{equation}
%whereas $\nor{u}^2_{H_0^1(U^L)}=\nor{u}^2_{H_0^1(U^R)}$ implies
%\begin{equation}\label{norma_H01}
%\sum_j \lambda_j^L{\alpha_j}^2 = \sum_k \lambda_k^R\left(\mathcal U
%^j_k \alpha_j\right)^2.
%\end{equation}
%If we restrict our analysis to $\nor{u}^2_{L^2(U^L)}=1$, choosing
%$\alpha_{\overline j}=(0,0,\ldots, 1, 0,0, \ldots)$ with 1 in the
%$\overline j$-th position, Equation \eqref{norma_L2} implies $\sum_k
%\left(\mathcal U ^{\overline j}_k \right)^2=1$, that is $\{\mathcal
%U ^j_k\}_k$ are coefficients of a convex linear combination for
%every $j$ fixed; whereas Equation \eqref{norma_H01} implies
%$\lambda_{\overline j}^L = \sum_k \lambda_k^R\left(\mathcal U
%^{\overline j}_k \right)^2$. Then,
%\begin{eqnarray*}
%\nor{u}^2_{H^{1/2}(U^L)}=\sum_j
%\left(\lambda_j^L\right)^{1/2}(\alpha_j)^2 = \sum_j \left(\sum_k
%\lambda_k^R\left(\mathcal U ^j_k \right)^2\right)^{1/2}(\alpha_j)^2
%\end{eqnarray*}

Then, $\mathcal U$ is a \emph{partially isometric operator} from
$\mathfrak{h}^{1/2}_R$ into $\mathfrak{h}^{1/2}_R$, since it is an
isometry on the subspace $\mathfrak{h}^{1/2}_L$. So, $\mathcal
U\mathcal U^*=\mathbb I$ on $\mathfrak{h}^{1/2}_L$ (see \cite{K}),
and multiplying the eigenvalue equation $(\mathcal
U^*\Lambda^R\mathcal U)\alpha=\mu\alpha$ by $\mathcal U$ we obtain
$$\Lambda^R\mathcal U\,\alpha = \mathcal U\,\mu\alpha=\mu\mathcal U\,\alpha,$$
the thesis.
\end{proof}
Thanks to the previous proposition, Equation
\eqref{saltoderivata_coefficienti_giusta_2} has a unique solution
which is nontrivial since $\alpha_0\neq 0$.

%Furthermore, we recall that from Remark \eqref{andamento_u^L,R} the
%asymptotic behavior of $u^L$ for $x\rightarrow -\infty$ is given by
%$\alpha_1$.
We note that whenever $\Phi$ is the solution to the problem
\eqref{problema_semicilindro}, then the first component $\alpha_1$
of the solution $\alpha$ to Equation
\eqref{saltoderivata_coefficienti_giusta_2} is for sure different
from zero. This is implied by the uniqueness of a
\underline{positive} solution to the problem \eqref{problema_omega}.
Moreover, from Remark \eqref{andamento_u^L,R} it describes the
asymptotic behavior of $u^L$ for $x\rightarrow -\infty$.

\subsection{Generalization to union of many chambers}\label{subsec:many_chambers}

Let us consider a domain which is a union of several different
chambers, such that the width of each chamber is negligible with
respect to the corresponding length. We mean $\Omega= C^1 \cup
\ldots \cup C^N$. The previous case $\Omega=C^L\cup C^R$ is
obviously covered by this type of domains. The proof of existence
and uniqueness of a $\mathcal C^1$ positive harmonic function in
such a domain is a straightforward consequence of Theorem
\eqref{teorema_omega}. As a matter of fact, we can merely iterate
its proof $N-1$ times with the suitable (slight) modifications,
where $N$ denotes the number of the chambers.

Moreover, suppose not to know the number of the chambers, but rather
the asymptotic behavior of the solution for $x\rightarrow -\infty$,
that is
\begin{equation}\label{asintotico_chambers}
u(x,y) \stackrel{x\rightarrow -\infty}{\sim} \kappa\
\e^{\sqrt{\lambda_1^1}x}\varphi_1^1(y)
\end{equation}
where $\lambda_1^1$ denotes the first eigenvalue for $\Delta^{N-1}$
for the first chamber and $\varphi_1^1(y)$ its relative
eigenfunction. Then it will be
\begin{equation}\label{kappa}
\kappa = \alpha_1^1 \cdot \alpha_1^2 \cdot \ldots \cdot
\alpha_1^{N-1},
\end{equation}
where $\alpha_1^j$ are the analogues of $\alpha_1$ in Equation
\eqref{sviluppi} for the couple of chambers $(C^j,C^{j+1})$. In this
way we can deduce the number of the chambers from $\kappa$, i.e.
from the solution's asymptotic behavior at $-\infty$.

Conversely, if the domain consists in the union of $N$ chambers, we
can immediately state that the asymptotic behavior of the unique
$C^1$ positive harmonic function for $x\rightarrow -\infty$ is
\eqref{asintotico_chambers} with $\kappa$ given by \eqref{kappa}.

%\section{Solutions with infinite energy in both directions}
%
%Following the proof of Proposition \eqref{esistenza_Omega}, we can
%construct a solution which has infinite energy in both directions.
%
%\begin{thm}
%There exists a positive constant $\gamma^*=\gamma^*(U_L,U_R)$ such
%that for every $\gamma\in(0,\gamma*)$ there exists a positive
%solution $u$ to the problem \eqref{problema_omega} such that the two
%following hold simultaneously
%\begin{itemize}
%\item $u$ is asymptotic to $\e^{\sqrt{\lambda_1^R}x}$ for $x\rightarrow
%+\infty$
%\item $u$ is asymptotic to $\gamma\e^{-\sqrt{\lambda_1^L}x}$ for $x\rightarrow
%-\infty$
%\end{itemize}
%being $\lambda_1^R$ and $\lambda_1^L$ the first eigenvalue of the
%$(N-1)$-dimensional Laplacian on $U_R$ and $U_L$ respectively.
%\end{thm}
%\begin{proof}
%We can follow the same outline of the proof of Proposition
%\eqref{esistenza_Omega}. If we impose the aforementioned asymptotic
%behavior as $x\rightarrow -\infty$, Equation
%\eqref{esistenza-unicita} becomes
%\begin{equation}
%(S_L-S_R)\varphi=2\sqrt{\lambda_1^R}\psi_1^R(y) - \gamma 2
%\sqrt{\lambda_1^L}\psi_1^L(y).
%\end{equation}
%The constant $\gamma^*$ in the statement will be
%$$ \gamma^*:= \sup\left\{\gamma>0 \text{ s.t. } 2\sqrt{\lambda_1^R}\psi_1^R(y) - \gamma 2
%\sqrt{\lambda_1^L}\psi_1^L(y) > 0 \text{ in $U_L\cup U_R$}\right\};
%$$ so that it depends only on the couple $(U_L,U_R)$.
%\end{proof}
%
%\begin{rem}
%If $U_R \supset U_L$ then $\gamma <1$.
%\end{rem}


\begin{thebibliography}{99}

\bibitem{AFT13} L. Abatangelo, V. Felli, S. Terracini, {\em On the sharp effect of attaching a thin handle
on the spectral rate of convergence}. Journal of Functional Analysis 266 (2014), 3632--3684.

\bibitem{AFT12} L. Abatangelo, V. Felli, S. Terracini, {\em
    Singularity of eigenfunctions at the junction of shrinking
    tubes. Part II}.  To appear in Journal of Differential Equations.

\bibitem{A} R.A. Adams, J.J.F. Fournier, {\em Sobolev Spaces},
Academic Press 2003.

\bibitem{BS} G. Buttazzo, F. Santambrogio, {\it Asymptotical
    compliance optimization for connected networks},
  Netw. Heterog. Media 2 (2007), no. 4, 761--777.

\bibitem{BSV} G. Buttazzo, F. Santambrogio, N. Varchon, {\it
    Asymptotics of an optimal compliance-location problem}, ESAIM
  Control Optim. Calc. Var. 12 (2006), no. 4, 752--769.

\bibitem{D77} B.E.J. Dahlberg, {\em Estimates for harmonic measure}, Arch. Rational Mech. Anal.  {\bf 65}  (1977),
no. 3, 275--288.

\bibitem{FT12} V. Felli, S. Terracini, {\it Singularity of
    eigenfunctions at the junction of shrinking tubes. Part I\/},
  J. Differential Equations,  255  (2013), n.4, 633-700.

\bibitem{GT} Gilbarg, Trudinger, {\em Elliptic partial differential
equations}, Springer 2001.

\bibitem{I} V. Isakov, {\em Inverse problems for Partial Differential Equations}, Springer 2006.

\bibitem{JK} D.S. Jerison, C.E. Kenig, {\em Boundary behavior of harmonic functions in nontangentially accessible domains\/}.
Adv. in Math.  {\bf 46}  (1982), no. 1, 80--147.

\bibitem{K} T. Kato, {\em Perturbation theory for linear operators},
Springer-Verlag 1995.

\bibitem{Ki} A. Kirsch, {\em An Introduction to the Mathematical Theory of Inverse Problems}, Springer, 1996.

\bibitem{Mu} M. Murata, {\em On construction of Martin boundaries for second order elliptic equations}, Publ. Res. Inst. Math. Sci. 26 (1990), no. 4, 585–627.

\bibitem{Pi88}Y. Pinchover, {\em
 On positive solutions of second-order elliptic equations, stability results, and classification},
 Duke Math. J., 57 (1988), pp. 955-980

\bibitem {Pi1988} Y. Pinchover, {\em On positive solutions of elliptic equations with periodic coefficients in unbounded domains}, Maximum principles and eigenvalue problems in partial differential equations (Knoxville, TN, 1987), 218–230, Pitman Res. Notes Math. Ser., 175, Longman Sci. Tech., Harlow, 1988

\bibitem{Pi94} Y. Pinchover,
{\em On positive Liouville theorems and asymptotic behavior of solutions of Fuchsian type elliptic operators},
Ann. Inst.H. Poincar\'e Anal. Non Lin\'eaire  11 (1994), no. 3, 313-341.

\bibitem{Pinsky} R. G. Pinsky,  {\em Positive harmonic functions and diffusion}, Cambridge Studies in Advanced
Mathematics, 45. Cambridge University Press, Cambridge, 1995.

\bibitem{PW} M. Protter, H. Weinberger, {\em Maximum Principles in Differential Equations}, Springer, 1984.

\bibitem{R} W. Rudin, {\em Real and Complex Analysis}, McGraw-Hill, 1987.

\end{thebibliography}
\end{document}